\documentclass[11pt,a4paper]{article}
\usepackage{amsmath,amsthm}
\usepackage{amsfonts,amstext,amssymb, mathtools, comment, undertilde}

\newcommand{\bs}{\boldsymbol}
\newcommand{\p}{\mathbb P}
\newcommand{\e}{\mathbb E}
\newcommand{\D}{\mathrm d}

\newcommand{\levy}{L\'{e}vy }
\newcommand{\cadlag}{c\`adl\`ag}
\newcommand{\C}{\mathbb C}

\newcommand{\R}{\mathbb R}
\renewcommand{\Re}{{\rm Re}}

\newcommand{\Cpos}{\mathbb C^{\Re>0}}

\newcommand{\diag}{\mathrm {diag}}
\newcommand{\ind}[1]{\mbox{\rm\large  1}_{\{#1\}}}

\renewcommand{\a}{\alpha}

\newcommand{\bv}{{\bs v}}

\newcommand{\1}{{\bs 1}}
\newcommand{\0}{{\bs 0}}

\newcommand{\eL}{{\bs L}}

\newcommand{\matI}{\mathbb{I}}

\newcommand{\oX}{{\overline X}}
\newcommand{\uX}{{\underline X}}

\newcommand{\Rl}{{R_*}}
\newcommand{\Ru}{{R^*}}

\newcommand{\tx}[1]{\tau^{\{#1\}}}
\newtheorem{theorem}{Theorem}
\newtheorem{lemma}[theorem]{Lemma}
\newtheorem{proposition}[theorem]{Proposition}
\newtheorem{corollary}[theorem]{Corollary}

\theoremstyle{remark}
\newtheorem{remark}{Remark}[section]

\begin{document}
\bibliographystyle{plain}
\title{Occupation densities in solving exit problems for\\ Markov additive processes and their reflections.}
\author{Jevgenijs Ivanovs\footnote{Eindhoven
University of Technology and University of Amsterdam, e-mail: jevgenijs.ivanovs@gmail.com}
\qquad Zbigniew Palmowski\footnote{University of Wroc\l aw,
e-mail: zpalma@math.uni.wroc.pl}}

\maketitle \abstract{This paper solves exit problems for spectrally negative Markov additive processes and their reflections. A so-called scale matrix, which is a generalization of the scale function of a spectrally negative \levy process, plays a central role in the study of exit problems. Existence of the scale matrix was shown in~\cite[Thm. 3]{palmowski_fluctuations}. We provide a probabilistic construction of the scale matrix, and identify the transform. In addition, we generalize to the MAP setting the relation between the scale function and the excursion (height) measure. The main technique is based on the occupation density formula
and even in the context of fluctuations of spectrally negative L\'{e}vy processes this idea seems to be new. Our representation of the scale matrix $W(x)=e^{-\Lambda x}\eL(x)$ in terms of nice probabilistic objects opens up possibilities for further investigation of its properties.}

\section{Introduction}

This paper solves exit problems for spectrally negative Markov Additive Processes (MAPs) and their reflections.
Before entering our discussion on this subject we shall
simply begin by defining the class of processes we intend to work
with. Let $(\Omega,\mathcal F,\mathbf F,\p)$ be a filtered probability space, with standard (right-continuous and augmented) filtration $\mathbf F=\{\mathcal F_t:t\geq 0\}$.
On this probability space consider a real-valued \cadlag\ (right-continuous with left limits) process $X=\{X(t):t\geq 0\}$
and a right-continuous jump process $J=\{J(t):t\geq 0\}$ with a finite state space $E=\{1,\ldots,N\}$,
such that $(X,J)$ is adapted to the filtration $\mathbf F$.
The process $(X,J)$ is a MAP if,
given $\{J(t)=i\}$, the pair $(X(t+s)-X(t),J(t+s))$ is independent
of $\mathcal F_t$ and has the same law as $(X(s)-X(0),J(s))$ given
$\{J(0)=i\}$ for all $s,t\geq 0$ and $i\in E$.
It is common to say that $X$ is an additive component and $J$ is a background process
representing the environment. Importantly, a MAP has a very special structure, which we reveal in the following. It is immediate that $J$ is a Markov chain. Furthermore, $X$ evolves as some spectrally negative \levy process $X_i$ while $J=i$.
In addition, a transition of $J$ from $i$ to $j\neq i$ triggers a jump of $X$ distributed as $U_{ij}\leq 0$, where $i,j\in E$. All the above components are independent.
This structure explains the other commonly used name for a MAP -
`Markov-modulated \levy process'.

The use of MAPs is widespread,
making it a classical model in applied probability with a variety of
application areas, such as queues, insurance risk, inventories, data
communication, finance, environmental problems and so forth, see \cite[Ch. XI]{asmussen:apq}, \cite[Ch. 7]{prabhu}, \cite{lambda, time_rev, MMBM_2sided, asm_avram_pist, breuer, extremes_dieker, Rogers2, breuer_MMBM}.

Throughout this paper we will consider only spectrally negative MAPs, where
the additive component $X$ has no positive jumps. We exclude the trivial cases when the paths of $X$ are a.s.\ monotone. It is noted that phase-type distributions fit naturally into the framework of MAPs. Positive phase-type jumps can be incorporated into the model using a standard trick, see e.g.~\cite{asm_avram_pist}. This is achieved by enlarging the number of states of the background process and replacing phase-type jumps by linear stretches of unit slope.

One of the main contribution of this paper is in identification, in Theorem \ref{thm:main_scale}, of the scale matrix~$W$
which plays the same role as the scale function in the theory of fluctuations of \levy processes (for a beautiful review on this subject see
\cite{KKR}).
Formally, the scale matrix is characterized by its Laplace transform:
\[\int_0^\infty e^{-\a x}W(x)\D x=F(\a)^{-1},\]
where $F(\alpha)$ is a matrix analogue of the Laplace exponent of a \levy process:
\[\e(e^{\a X(t)};J(t)=j|J(0)=i)=\left(e^{F(\a)t}\right)_{ij}.\]
This result generalizes \cite[Thm. 3]{palmowski_fluctuations}, where
only existence of a scale matrix was proven. In the case of spectrally negative \levy process this result
follows contributions of \cite{Zolotarev,
Takacs, Emery, Bingham, Rogers2, Bertoin6a, Bertoin7}, see also \cite{kyprianou, bertoin} for an overview.

In the basis of the above transform lies a probabilistic construction of the scale matrix~$W(x)$. It is given in~(\ref{eq:W}), and involves the first hitting time of a level $-x$.
It can be seen as a counterpart of the representation of the scale function of a \levy process in terms of the potential density at~$-x$, see~\cite[Thm. 1(ii)]{Pistorius2}. It turns out that our representation can be written in a more concise way:
\[W(x)=e^{-\Lambda x}\eL(x),\]
where $\Lambda$ is a transition rate matrix of the Markov chain associated with the first passage, and $\eL(x)$ is a matrix of expected occupation times at $0$ up to the first passage time over~$x$. This allows to study $W(x)$ by analyzing $\eL(x)$, which has a very nice probabilistic interpretation.
For example, we can immediately deduce that the entries of the matrix $e^{\Lambda x}W(x)$ are non-negative, increasing functions.
Considering the domain of \levy processes, it is noted that the function $\eL(x)$ plays an important role in the study of local times, see~\cite[Ch. V.3]{bertoin}. Furthermore, \cite[Lem. V.11]{bertoin} determines this function as a doubly indefinite integral, which opens possibilities for new numerical identification of the scale function. We give some further comments on these issues in Remark~\ref{link}.

It is well-known in the domain of \levy processes that the intensity of the Poisson process (indexed by local time) counting the number of excursions with height exceeding $a$ is given by $W'_+(a)/W(a)$, see \cite[p. 195]{bertoin}. This result is very important for analyzing the smoothness of the scale function, see \cite[Lem. 2.3]{Mladen}. We generalize this result to MAPs in Theorem \ref{thm:derivatives}.

Furthermore, we solve exit problems for $X$, as well as $-X$, reflected at their infima. The corresponding results are given in Theorem~\ref{thm:reflected} and Theorem~\ref{thm:reflected2}.
They generalize results in \cite{avram_kypr_pist, Yor, pist_exit} where spectrally one-sided \levy processes were analyzed.
The solutions to these exit problems contain the second scale matrix $Z$ defined formally in (\ref{matrixZ}).

Finally, we propose new, unified technique for deriving
fluctuation identities for MAPs (and \levy process in particular) and their reflections. This technique is based on the properties of occupation times jointly with
occupation density formula. The main proofs in this paper rely on these arguments which,
we believe, even in the context of spectrally one-sided \levy processes are new.
For \levy processes this technique complements the use of martingales (see \cite{kyprianoupalmowski, Yor, avram_kypr_pist}), potential densities (see \cite{Pistorius2}), and excursion calculations (see \cite{Doney, pist_exit}).

The paper is organized as follows. In Section~\ref{sec:prelim} we recall some basic definitions and properties of MAPs. In Section \ref{sec:main} we present our main results.
In Section \ref{sec:occ_density} we introduce basic facts concerning occupation densities for MAPs.
In Section \ref{sec:scale} we identify the scale matrix proving Theorem
\ref{thm:main_scale}. In Section \ref{sec:2_sided_reflection}
we focus on some identities for two-sided reflection of a MAP. These identities are fundamental in solving exit problems for one-sided reflections, which are
analyzed in Section \ref{sec:further}. Finally, Appendix contains a few technical lemmas.


\section{Preliminaries}\label{sec:prelim}

Let $(X,J)$ be a spectrally negative MAP.
Let us use $\p_{x,i}$ with $x\in\R,i\in E$ to denote the law of $(X,J)$ given
$\{X(0)=x,J(0)=i\}$. We often write $\p_i$ for $\p_{0,i}$.
As much as possible we shall prefer to work with
matrix notation. For a random variable $Y$ and (random) time $\tau$, we shall understand
$\e_x[Y; J(\tau)]$ as the matrix with $(i,j)$-th element $\e_{x,i}(Y\ind{J(\tau)=j})$, where $\ind{A}$ denotes the indicator function of an event $A$. We also write
$\p_{x}[A, J(\tau)]$ for a matrix with elements $\p_{x,i}(A,J(\tau)=j)$.

Recall that $X$ evolves as some \levy process $X_i$ when $J(t)=i$.
The Laplace exponent of $X_i$ is denoted through $\psi_i(\a),\a\geq 0$: $\e e^{\a X_i(t)}=e^{\psi_i(\a)t}$. In order to avoid technicalities we assume that none of $X_i$ is a process with non-increasing paths. That is every $X_i$ is a spectrally negative \levy process. This is a rather mild assumption, because a MAP can be time-changed so that such non-increasing stretches contribute to negative jumps; often the results concerning the original process can be recovered from this time-changed process. Essentially without loss of generality it is assumed that $J$ is
irreducible. Let $Q$ be its $N\times N$ transition rate matrix.
For all $\a\geq 0$ define an $N\times N$ matrix:
\[F(\a)=\diag(\psi_1(\a),\ldots,\psi_N(\a))+Q\circ
G(\a),\] where $G(\a)_{ij}=\e e^{\a U_{ij}}$ and $A\circ
B=(a_{ij}b_{ij})$ stands for entry-wise (Hadamard) matrix product.
The matrix-valued function $F(\a)$ is referred to as the
`matrix exponent' of a MAP. It characterizes the law of $(X,J)$, and can be seen as a multi-dimensional analogue of the Laplace exponent of a \levy process.

It is convenient to work with killed processes. We enlarge the state space of $J$ to $E\cup\{\partial\}$, where $\partial$ is an additional absorbing state. Furthermore,
we put $X=\infty$ whenever $J=\partial$. It is known that a MAP satisfies associated strong Markov property, that is, we can replace $t$ with a stopping time $T$ in the definition of a MAP.

As a consequence of Perron-Frobenius theory $F(\a)$ has a real simple eigenvalue $k(\a)$, which is larger than the real part of any other eigenvalue. Using the standard facts from matrix analysis, see~\cite{matrix_analysis}, one can show that $k(0)\leq 0$; it is 0 if and only if $Q\1=\0$ (the case of no killing). In the case of no killing the real right derivative $k'(0)$ gives the asymptotic drift:
\[k'(0)=\lim_{t\rightarrow\infty} X(t)/t\quad \p_i\text{-a.s.\ for all }i\in E.\]

The first passage time of $X$ over level $\pm x$, where $x\geq 0$, is defined through
\[\tau^\pm_x=\inf\{t\geq 0: \pm X(t)> x\}.\]
Observe that $X(\tau^+_x)=x$,
because of the absence of positive jumps. The strong Markov
property of $X$
implies that $(\tau^+_x,J(\tau^+_x)),x\geq 0$ is a MAP killed at $\oX=\sup\{X(t):t\geq 0\}$.
We will often use the transition rate matrix
$\Lambda$ of the Markov chain $J(\tau^+_x),x\geq 0$:
\begin{equation}\label{lambda}\p[J(\tau_x^+)]=e^{\Lambda x}.\end{equation}

The matrix $\Lambda$ is closely related to the matrix exponent $F(\a)$. In particular, the non-zero eigenvalues of $-\Lambda$ coincide with the zeros of $\det(F(\a))$ in $\Cpos=\{z\in \C:\Re(z)>0\}$, see~\cite[Theorem 1]{lambda}. In addition, $\Lambda$ has a zero eigenvalue if and only if $\oX=\infty$ a.s., which is equivalent to $Q\1=\0$ (no killing) and $k'(0)\geq 0$.
For discussion concerning $\Lambda$ and its identification through a certain matrix integral equation
see \cite{rogers, asmussen_fluid_flow, pistorius, miyazawa, extremes_dieker, breuer, lambda, time_rev}.

\begin{remark}
Let us point out some advantages of working with killed processes.
Let $e_q$ be an independent exponential random variable of rate $q>0$. Suppose we send $J$ to the absorbing state $\partial$ at $t=e_q$ (it could have been killed before).
By doing so we get a Markov chain with transition rate matrix $Q-q\matI$, which leads to a MAP with the matrix exponent $F(\a)-q\matI$. Indeed, \[\e[e^{\a X(t)};t<e_q,J(t)]=e^{-qt}\e [e^{\a X(t)};J(t)]=e^{(F(\a)-q\matI)t}.\]
Similarly,
\[\e[e^{-q\tau_x^+};J(\tau_x^+)]=\p[\tau_x^+<e_q,J(\tau_x^+)]=e^{\Lambda^qx},\]
where $\Lambda^q$ corresponds to the matrix exponent $F^q(\a)=F(\a)-q\matI$.
It is then clear that $-\Lambda^q$ reduces to the right inverse $\Phi(q)$ of the Laplace exponent in the case of a \levy process. Observe that a matrix of probabilities $\p[J(\tau_x^+)]$ becomes a matrix of Laplace transforms $\e[e^{-q\tau_x^+};J(\tau_x^+)]$ if one superimposes exponential killing of rate $q>0$.
Throughout this work we keep killing implicit, which greatly simplifies the presentation.
\end{remark}

\section{The main results}\label{sec:main}
Let us present our main results. We stress that the following statements are given in their concise form. That is, one can always superimpose exponential killing of rate $q>0$ to add the random time appearing in $J(\cdot)$ to the transform, see also comments following Theorem~\ref{thm:main_scale}.
The main result of this paper concerns the two-sided upward problem (regarding \levy processes see \cite[Thm. VII.8]{bertoin}, \cite[Thm. 8.1]{kyprianou} and the references given in the latter book).

\begin{theorem}\label{thm:main_scale}
There exists a unique
continuous function $W:[0,\infty)\rightarrow\R^{N\times N}$
such that $W(x)$ is invertible for all $x>0$,
\[\p[\tau_a^+<\tau^-_b,J(\tau_a^+)]=W(b)W(a+b)^{-1}\text{ for all }a,b\geq 0\text{ with }a+b>0,\]
and
\[\int_0^\infty e^{-\a x}W(x)\D x=F(\a)^{-1}\]
for all $\a>\eta=\max\{\Re(z):z\in\C,\det(F(z))=0\}$.
In addition,
\begin{equation}\label{localrepr}W(x)=e^{-\Lambda x}\eL(x),\end{equation} where
$\eL(x)$ is a matrix of expected occupation times at $0$ up to the first passage over $x$ defined in~(\ref{eq:H}).
\end{theorem}

Theorem~\ref{thm:main_scale} is proven in Section~\ref{sec:scale}.
Recall that it is enough to superimpose exponential killing of rate $q>0$ to obtain the formula for $\e[e^{-q\tau_a^+};\tau_a^+<\tau^-_b,J(\tau_a^+)]$.
We define $\eL(x)$ and prove some properties of it in Section \ref{sec:occ_density}. Note that $\eL(x)$ tends to $\eL$, the matrix of expected occupation times at 0, as $x\rightarrow\infty$. We prove in Lemma \ref{lem:L_transform}
that $\eL$ has finite entries and is invertible unless $Q\1=\0$ and $k'(0)=0$ (the recurrent case).

The scale matrix can be given in an explicit form in certain cases, such as in the case of a Markov-modulated Brownian motion (MMBM), where some variance parameters are allowed to be $0$, see \cite{pist_jiang, MMBM_2sided, MMBM_bernardo}. This result can be used to find the steady-state buffer content in a fluid model driven by Markov sources, see
\cite{anick_mitra_sondhi} and a survey \cite{fluid_survey}. Furthermore, a spectrally negative MAP with all the jumps of phase-type can be embedded into an MMBM. The scale matrix of this MMBM can be easily converted to the scale matrix of the original process.
Many examples concerning scale functions of \levy processes can be found in \cite{kypr_hubalek}.

\begin{remark}\label{link}
In the context of spectrally negative \levy processes the representation (\ref{localrepr}) produces interesting links with known representations (see e.g.
\cite[eq. (2.26)]{KKR} and \cite[Thm. 1(ii)]{Pistorius2}).
To see it, recall that $\Lambda=-\Phi$ where $\Phi$ is the right-most zero of the Laplace exponent. In fact, $\Phi$ depends on the killing rate $q\geq 0$ which we keep implicit throughout this paper.
First we consider the representation of $W(x)$ in terms of the potential density $u(\cdot)$ of $X$ as given in~\cite{Pistorius2}. Namely, combining (i) and (ii) of~\cite[Thm. 1]{Pistorius2} we get for any $q>0$ the following
\[W(x)=e^{\Phi x}[u(0)-u(x)u(-x)/u(0)].\]
It is not difficult to show that for $q>0$ the expression in the brackets is exactly $\eL(x)$, see also~\cite[Eq. (V.18)]{bertoin}.
Another representation of $W(x)$ in terms of the potential measure of the descending ladder height process $\hat H(t)$ is given in~\cite[eq. (2.26)]{KKR}. Comparing it to (\ref{localrepr}) we get
\[\eL(x)=\int_0^x e^{-\Phi x}\int_0^\infty \D t\p(\hat H(t)\in \D y)\] for at least $q=0$.
It would be interesting to provide a direct probabilistic proof of this identity.
Some further identities involving $\eL(x)$ can be found in~\cite[Ch. V.3]{bertoin}

\end{remark}

Define now the reflection of $X$ at 0 (starting from an arbitrary $x\geq 0$):
$$Y(t)=X(t)+R(t),$$ where $R(t)=-(0\wedge\uX(t))$ is a regulator at $0$. Similarly define the reflection of the dual $-X$ at 0: $\hat Y(t)=-X(t)+\hat R(t)$, so that $\hat R(t)=(0\vee\oX(t))$.
First passage times of the reflected processes are given by
\begin{align*}&T_a=\inf\{t\geq 0:Y(t)>a\},&\hat T_a=\inf\{t\geq 0:\hat Y(t)>a\}.\end{align*}

In solving exit problems for the reflected processes the crucial is the second scale matrix:
\begin{equation}\label{matrixZ}
Z(\a,x)=e^{\a x}\left(\matI-\int_0^xe^{-\a y}W(y)\D yF(\a)\right)\text{ for }\a,x\geq 0.\end{equation}
Note that $Z(\a,x)$ is continuous in $x$ with
$Z(\a,0)=\matI$, and is analytic in $\a\in\Cpos$. In the case of a
single background state and independent killing of rate $q\geq 0$ we obtain $Z^q(0,x)=1+q\int_0^xW^q(y)\D y$,
which is a common definition of the $Z$ function corresponding to a
spectrally negative \levy process, see, for example,
\cite[p. 214]{kyprianou}.
The next theorem identifies, for a reflected process $Y$, the joint transform of the first passage time  $T_a$ and the value $R(T_a)$ of the regulator at this time.

\begin{theorem}\label{thm:reflected}For all $\a\geq 0$ and $a>0$ it holds that $Z(\a,a)$ is invertible. Furthermore,
\[\e_{x}[e^{-\a R(T_a)};J(T_a)]=Z(\a,x)Z(\a,a)^{-1},\]
where $x\in [0,a]$.
\end{theorem}

This theorem immediately results in the corollary, which solves downward two-sided exit problem for $X$ (see again \cite[Thm. 8.1]{kyprianou}.

\begin{corollary}For any $\a\geq 0,a>0$ and $x\in [0,a]$ it holds that
\[\e_{x}[e^{\a X(\tau_0^-)};\tau_0^-<\tau_a^+,J(\tau_0^-)]=Z(\a,x)-W(x)W(a)^{-1}Z(\a,a).\]
\end{corollary}
\begin{proof}
Using the law of total probability and the strong Markov property of $X$ write
\begin{align*}\e_{x}[e^{-\a R(T_a)};J(T_a)]=\p_{x}[\tau^+_a<\tau_0^-,J(\tau^+_a)]+\e_{x}[e^{\a X(\tau_0^-)};\tau^-_0<\tau^+_a,J(\tau_0^-)]\e[e^{-\a R(T_a)};J(T_a)].
\end{align*}
Use the identities of Theorem~\ref{thm:main_scale} and Theorem~\ref{thm:reflected} to complete the proof.
\end{proof}
Let us derive the identity for the first passage of $X$ over a negative level.

\begin{corollary}Assume it is not true that $Q\1=0$ and $k'(0)=0$.
Then for all $x,\a\geq 0$,
\[\e_{x}[e^{\a X(\tau_0^-)};J(\tau_0^-)]=Z(\a,x)-W(x)\eL^{-1}(\a\matI+\Lambda)^{-1}\eL F(\a).\]
If $\a$ is a zero of $\det(F(\a))$ then it is a zero of $\det(\a\matI+\Lambda)$ and the expression on the right should be interpreted in a limiting sense.
\end{corollary}
\begin{proof}
According to Theorem~\ref{thm:main_scale} for $\a>\eta$ we have
\[Z(\a,a)=e^{\a a}\int_a^\infty e^{-\a y}W(y)\D yF(\a)=\int_0^\infty e^{-\a y}W(y+a)\D yF(\a).\]
Recall that $W(a)=e^{-\Lambda a}\eL(a)$. Thus:
\[W(a)^{-1}Z(\a,a)=\eL(a)^{-1}\int_0^\infty e^{(-\a\matI-\Lambda)y}\eL(a+y)\D yF(\a).\]
Applying the dominated convergence theorem we derive
\[\lim_{a\rightarrow\infty}W(a)^{-1}Z(\a,a)=\eL^{-1}\int_0^\infty e^{(-\a\matI-\Lambda)y}\D y \eL F(\a)\]
which proves the assertion of the theorem for $\a>\eta$. Analytic continuation completes the proof.
\end{proof}

Assume that $X(0)=0$. Then $\hat R(t)=\oX(t)$ is the running supremum of $X$.
Consider the process $-\hat Y(t)=X(t)-\oX(t)$ depicting excursions of $X$ from its supremum.
 Suppose we kill the Markov chain $\{J(\tau^+_x),x\geq 0\}$
 upon arrival of the first excursion (from the maximum) exceeding a certain height $a>0$. Note that this excursion arrives at $\zeta_a=\hat R(\hat T_a)$ where
 we measure (local) time along the $x$ coordinate.
 The strong Markov property of $X$ shows that this killed Markov chain is again a Markov chain. Let us denote its transition rate matrix by $\Lambda_{[0,a]}$:
 \begin{align}\label{eq:Lambda_a}\p[x<\zeta_a,J(\tau_x^+)]=e^{\Lambda_{[0,a]}x}\text{ for all }x\geq 0.\end{align}
Similarly, one can kill $J(\tau_x^+)$ upon arrival of the first excursion of height at least $a$, where $a>0$. The corresponding transition probability matrix is denoted by $\Lambda_{[0,a)}$. These two transition probability matrices are closely related to the scale matrix $W$ and its one-sided derivatives.
\begin{theorem}\label{thm:derivatives}
For all $a>0$ the limits
\begin{align*}&W'_+(a)=\lim_{\epsilon\downarrow 0}(W(a+\epsilon)-W(a))/\epsilon, &
W'_-(a)=\lim_{\epsilon\downarrow 0}(W(a)-W(a-\epsilon))/\epsilon
\end{align*}
 exist. In addition, $\Lambda_{[0,a]}=W'_+(a)W(a)^{-1}$ and $\Lambda_{[0,a)}=W'_-(a)W(a)^{-1}$.
\end{theorem}
The above theorem shows that $W(a),a>0$ is differentiable if $\Lambda_{[0,a]}$ and $\Lambda_{[0,a)}$ coincide.
For this to hold it is sufficient to require that every
$X_i$ is of unbounded variation.
Theorem \ref{thm:derivatives} is strongly related to fluctuation theory and excursion measure
(compare with \cite[Lem. 2.3]{Mladen} and \cite[Lem.~1]{pist_exit}).

Let us finally, identify, for the reflected process $\hat Y$ with $\hat Y(0)=x\in[0,a]$, the joint transform of the first passage time $\hat T_a$, the value $\hat R(\hat T_a)$ of the regulator at this time, and the overshoot $\hat Y(\hat T_a)-a$.
In theory of \levy processes
this seminal result is due to
\cite{avram_kypr_pist} (see also e.g. \cite{Yor, pist_exit, maria2, bekker, boxma}).
\begin{theorem}\label{thm:reflected2}
It holds for all $\a,\theta\geq 0$ and $x\in[0,a]$ that
\begin{align*}
&\e_{\hat Y(0)=x} [e^{-\theta
\hat R(\hat T_a)-\a
(\hat Y(\hat T_a)-a)};J(\hat T_a)]\\&=Z(\a,a-x)+W(a-x)[W'_+(a)+\theta
W(a)]^{-1}[W(a)F(\a)-(\a+\theta)Z(\a,a)].
\end{align*}
\end{theorem}

\section{Occupation densities and local times}\label{sec:occ_density}
It has been noted above that occupation densities play a fundamental role in our construction of the scale matrix. In this section we first review the concept of occupation densities of a \levy process. Then we generalize it to the case of MAPs. Before we proceed let us define the \emph{first hitting time} of a level $x$:
\[\tx{x}=\inf\{t>0:X(t)=x\}.\] One can show following the steps of~\cite[Corollary~I.8]{bertoin} that $\tx{x}$ is a stopping time.

\subsection*{\levy processes}
Throughout this subsection we assume that $X$ is a \levy process with \levy measure $\nu$; we do not impose any restrictions on $X$ yet (it may have positive jumps in particular).
The \levy measure satisfies $\int_\R(1\wedge x^2)\nu(\D x)<\infty$. If $X$ is of bounded variation then
necessarily $\int_\R(1\wedge |x|)\nu(\D x)<\infty$. In this case $X$ equals in law to the process $dt+\sum_{0\leq s\leq t}\Delta(s)$, where $\Delta(t)$ is a Poisson point process with characteristic measure $\nu$, for some $d\in\R$ which is known as the \emph{drift coefficient}. Furthermore, a bounded variation process is a compound Poisson process (CPP) if and only if $\nu$ has finite mass and $d=0$.
These and other facts can be found in the book~\cite{bertoin}.

In the following we assume that
\begin{equation}\label{eq:assumption}\p(\tx{x}<\infty)>0\text{ for all }x\in\R.\end{equation}
For this to hold it is sufficient to require, see~\cite[Thm.~1]{kesten},
that $X$ is not a bounded variation process with $d=0$, $X$ is not monotone,
and $\nu(0,\infty)<\infty$.
The Assumption~(\ref{eq:assumption}) is sufficient, as stated by \cite[Prop.~II.16,
Thm.~V.1]{bertoin}, for the existence of the occupation density
$L(x,t)$ satisfying the \emph{occupation
density formula}
\begin{equation}\label{eq:occ_density}\int_0^tf(X(s))\D s=\int_\R f(x)L(x,t)\D x\end{equation}
for all measurable functions $f\geq 0$ a.s. and any fixed $t\geq 0$,
see~\cite[Eq.~V.(2)]{bertoin}. When $L(x,t)$ is seen as a function of $t$, we call it \emph{occupation time} at the level $x$.
Formally $L(x,t)$ could be defined in the following way.

Suppose that 0 is regular for itself. Then for every $x\in\R$
\[L^\epsilon(x,t)=\frac{1}{2\epsilon}\int_0^t\ind{|X(s)-x|<\epsilon}\D s\]
converges uniformly on compact intervals of time $t$, in $L^2(\p)$
as $\epsilon\downarrow 0$, see~\cite[Prop.~V.2]{bertoin}. This
limit is denoted through $L(x,t)$, which consequently is continuous
in $t$ a.s.

If 0 is irregular for itself then $X$ is of bounded variation, see~\cite[Thm.~8]{bretagnolle} and recall that (\ref{eq:assumption}) is true. In this case we define:
\[L(x,t)=|d|^{-1}N(x,t),\] where
$N(x,t)=\#\{s\in[0,t):X(s)=x\}$, see also~\cite{fitzsimmons}.

It can be shown that for each $x\in\R$, the occupation time $L(x,\cdot)$ is
an increasing $\mathcal F_t$-adapted process, which increases only
if $X=x$. Moreover, for every $y$ and every stopping time $T$ with
$X(T)=y$ on $\{T<\infty\}$, the shifted process
$L(x,T+t)-L(x,T),t\geq 0$ is independent of $\mathcal F_T$ under
$\p(\cdot|T<\infty)$, and has the same law as $L(x-y,t),t\geq 0$
under $\p$.
This property shows that if $L(0,\cdot)$ is continuous then it, in fact, coincides up to
a constant factor with the \emph{local time} of $X(t)$ at~0,
see~\cite[Prop.~V.4]{bertoin}.

Finally, the following proposition will be important later on.
\begin{proposition}\label{prop:hit_pass}
Assume that $X$ is not a CPP.
Then $\p(\tx{-x}<\tau_x^-)=0$ for any $x\geq 0$.
\end{proposition}
\begin{proof}
It is known that the local extrema of $X$ are all distinct,
except in the CPP case, see~\cite[Prop.~VI.4]{bertoin}. So
$\tx{0}<\tau_0^-$ has zero probability. Suppose next that for some $x>0$ the event $\tx{-x}<\tau_x^-$ has positive probability. Consider an independent CPP $Y$,
which has deterministic jumps of size $x$. Note that the \levy process $X+Y$ can achieve
its infimum without immediately passing it with positive
probability (there is exactly one jump of $Y$ in the interval $[0,\tx{-x}]$; note that $\tx{-x}>0$ a.s., because the paths of $X$ are \cadlag). Hence it is a CPP, which further implies that $X$ is a
CPP too, which is not the case.
\end{proof}

\subsection*{MAPs}
Let $(X,J)$ be a MAP satisfying the assumptions of Section~\ref{sec:prelim}.
The structure of MAP allows for immediate generalization of
the concept of occupation densities.
It is not difficult to show that the process $X$ observed at the times when $J=i$ is a \levy process. It can be seen as an independent sum of $X_i$ and a CPP. The latter process incorporates the increments of $X$ corresponding to the time intervals when $J\neq i$. These increments need not to be negative.
Importantly, this \levy process satisfies the assumption~(\ref{eq:assumption}); one can check the conditions following~(\ref{eq:assumption}).

Let us define $L(x,j,t)$, the
occupation density of $(X,J)$ at $(x,j)$ up to time $t$. If $0$ is irregular for itself we put
\[L(x,j,t)=\frac{1}{d_j}\#\{s\in[0,t):X(s)=x,J(s)=j\},\]
where $d_j>0$ is the drift coefficient of $X_j$;
otherwise $L(x,j,t)$ is the limit in $L^2(\p)$ of
\[L^\epsilon(x,j,t)=\frac{1}{2\epsilon}\int_0^t\ind{|X(s)-x|<\epsilon,J(s)=j}\D s\]
as $\epsilon\downarrow 0$.
Moreover, the occupation density formula becomes
\[\int_0^tf(X(s))\ind{J(s)=j}\D s=\int_\R f(x)L(x,j,t)\D x.\]

Observe that $L(x,j,t)$ is an increasing $\mathcal F_t$-adapted
process, which increases only if $X=x,J=j$. The strong Markov property of $X$ implies the following result.
\begin{proposition}\label{prop:occ_density_MAP}
For every $y,i$ and every stopping time $T$ with
$X(T)=y$ on $\{J(T)=i\}$, the shifted process
$L(x,j,T+t)-L(x,j,T),t\geq 0$ is independent of $\mathcal F_T$
under $\p(\cdot|J(T)=i)$, and has the same law as
$L(x-y,j,t),t\geq 0$ under $\p_i$.
\end{proposition}
This proposition immediately results in the following identity for $x\neq 0$:
\begin{equation}\label{eq:Lx_L}\e L(x,j,\infty)=\p[J(\tx{x})]\e
L(0,j,\infty).
\end{equation}
This identity would be true for all $x$ if we replaced the first hitting time with the first entrance time.
Let $\eL$ be a
$N\times N$ matrix with $(i,j)$-th element equal to $\e_i
L(0,j,\infty)$. Observe that $\eL$ has
strictly positive entries.

\begin{theorem}\label{thm:transform}
For all $\a\geq 0$, such that $k(\a)<0$, it holds that
\[\int_\R e^{\a x}\p[J(\tx{x})]\D x\eL=-F(\a)^{-1}.\]
\end{theorem}
\begin{proof}
Firstly, by the monotone convergence theorem we get
\[\e\lim_{t\rightarrow\infty}\int_0^t e^{\a X(s)}e_{J(s)}\D s=\int_0^\infty e^{F(\a)s}\D s=-F(\a)^{-1},\]
where we used the fact that the real parts of the eigenvalues of
$F(\a)$ are all negative.
Using the occupation density formula we obtain
\[\lim_{t\rightarrow\infty}\int_0^t e^{\a X(s)}\ind{J(s)=j}\D s=\lim_{t\rightarrow\infty}\int_\R e^{\a x}L(x,j,t)\D x=\int_\R e^{\a x}L(x,j,\infty)\D x.\]
Identity~(\ref{eq:Lx_L}) completes the proof.
\end{proof}

\begin{remark}\label{po}
If there is an $\a$ satisfying the conditions of
Theorem~\ref{thm:transform} then both $\eL$ and $\int_\R e^{\a
x}\p[J(\tx{x})]\D x$ must have finite entries. Furthermore, $\eL$ is
then invertible. This is the case if $Q$ is transient: $Q\1\neq \0$, because then $k(0)<0$.
\end{remark}

Let us introduce
another matrix $\eL(x),x\geq 0$, which will play an important role
in the study of the scale matrix. Let
\[\eL_{ij}(x)=\e_i L(0,j,\tau^+_x)\] for $x>0$.
That is, $\eL(x)$ is the expected occupation time at 0
up to the first passage time over $x$.
Clearly, the elements of $\eL(x)$ are non-negative, non-decreasing functions with $\lim_{x\rightarrow\infty}\eL(x)=\eL$. In addition, we define $\eL(0)=\lim_{x\downarrow 0}\eL(x)$.
But $\tau_x^+\downarrow 0$ a.s. as $x\downarrow 0$. Hence
\begin{equation}\label{eq:L0}\eL_{ij}(0)=\e_i L(0,j,0+)=\begin{cases}1/d_i&\text{if }i=j\text{ and }X_i\text{ is of bounded variation,}\\
0&\text{otherwise.}\end{cases}\end{equation}
Hence $\eL(x)$ is the expected occupation time at 0 up to \emph{and including} first passage over $x\geq 0$.
In view of
Proposition~\ref{prop:occ_density_MAP} the matrix $\eL(x),x\geq 0$ can
be expressed as
\begin{equation}\label{eq:H}\eL(x)=\eL-\p[J(\tau_x^+)]\p[J(\tx{-x})]\eL,\end{equation}
given that $\eL$ has finite entries.

Let us next present a technical lemma.
\begin{lemma}\label{lem:L_transform}
For any $u>0$ it holds that $\int_0^\infty e^{-ux}\eL(x)\D
x$ has finite entries. Moreover, matrix $\eL$ has finite entries unless $Q\1=\0$ and $k'(0)=0$.
\end{lemma}
\begin{proof}
See Appendix.
\end{proof}
The condition $Q\1=\0$ and $k'(0)=0$ excludes recurrent processes.
The above lemma implies that the entries of $\eL(x)$ are finite. Furthermore, they are continuous functions of $x\geq 0$. To see this use the definition of $\eL(x)$ and a.s.~continuity of $\tau_x^+$ at a fixed $x> 0$.

\section{Two-sided exit and the scale matrix}\label{sec:scale}
One of the key ingredients in the construction of a scale matrix is given by the following lemma.
\begin{lemma}For any $a,b\geq 0$ with $a+b>0$ the events $\{\tau_a^+<\tau_b^-\}$ and $\{\tau_a^+<\tx{-b}\}$ coincide a.s.
\end{lemma}
\begin{proof}
If $\tau_a^+<\tau_b^-$ then $\tau_a^+<\tx{b-}$, because $\tau_b^-\leq\tx{-b}$ a.s., which follows immediately from Proposition~\ref{prop:hit_pass}. It is left to show that $\tau_b^-\leq\tau_a^+$ implies $\tx{-b}\leq\tau_a^+$.
If $\tau_b^-<\tau_a^+$ then $\tx{-b}\leq \tau_a^+$, because of the absence of positive jumps (the level $-b$ should be hit on the way to the level $a$); the equality arises if $X$ is killed before reaching $\tx{-b}$. If $\tau_b^-=\tau_a^+$ then $\tau_a^+=\infty$ and hence $\tx{-b}\leq \tau_a^+$.
\end{proof}

The above lemma allows us to study $\p[\tau_a^+<\tx{-b};J(\tau^+_a)]$ instead of $\p[\tau_a^+<\tau_b^-;J(\tau^+_a)]$. The advantage of using the stopping time $\tx{-b}$ instead of $\tau^-_b$ is that the value of $X$ is known at the first time, but not at the second. Pick arbitrary $a,b\geq 0$ such
that $a+b>0$ and consider the following equations, which are an
immediate consequence of the strong Markov property:
\begin{align*}
\p[J(\tau^+_a)]&=\p[\tau_a^+<\tx{-b};J(\tau^+_a)]+\p[\tx{-b}<\tau_a^+;J(\tx{-b})]\p[J(\tau^+_{a+b})],\\
\p[J(\tx{-b})]&=\p[\tx{-b}<\tau_a^+;J(\tx{-b})]+\p[\tau_a^+<\tx{-b};J(\tau_a^+)]\p[J(\tx{-a-b})].
\end{align*}
Right-multiplying the second equation by $\p[J(\tau^+_{a+b})]=e^{\Lambda(a+b)}$ and subtracting the first one gives:
\begin{align*}&\p[\tau_a^+<\tx{-b};J(\tau^+_a)](\p[J(\tx{-a-b})]e^{\Lambda(a+b)}-\matI)\\&=\p[J(\tx{-b})]e^{\Lambda(a+b)}-e^{\Lambda a}.\end{align*}
Hence
\begin{equation}\label{eq:exit}\p[\tau_a^+<\tx{-b};J(\tau^+_a)]W(a+b)=W(b),\end{equation} where we define $W(x)$ for all $x\geq
0$ through
\begin{equation}\label{eq:W}W(x)=\left(e^{-\Lambda
x}-\p[J(\tx{-x})]\right)\eL,\end{equation}
and $\eL$ is the matrix of expected occupation times at 0 defined in Section~\ref{sec:occ_density}. Note that~(\ref{eq:exit}) holds true if $\eL$ is replaced by an arbitrary constant matrix. The importance of our particular choice will be made clear in a moment. We note that (\ref{eq:W}) is well defined whenever $\eL$ has finite entries, which is the case unless there is no killing $Q\1=\0$ and the process $X$ oscillates: $k'(0)=0$, see Lemma~\ref{lem:L_transform}. The later delicate case will be treated using a limiting argument. It can be shown that in this case the matrix in the brackets in~(\ref{eq:W}) is singular and the entries of $\eL$ are infinite.

The following lemma identifies the transform of $W(x)$ in the case when killing is present. Recall also the definition of $\eta$ as given in Theorem~\ref{thm:main_scale}.
\begin{lemma}\label{lem:transform}
If $Q\1\neq \0$ then for all $\a>\eta$ it holds that:
\[\int_0^\infty e^{-\a x}W(x)\D x=F(\a)^{-1}.\]
\end{lemma}
\begin{proof}
Using~(\ref{eq:W}) we write
\[\int_0^\infty e^{-\a x}W(x)\D x=\int_0^\infty e^{(-\Lambda-\a\matI)x}\D x\eL-\int_0^\infty e^{-\a x}\p[J(\tx{-x})]\D x\eL.\]
Recall that the set of eigenvalues of $-\Lambda$ coincides with
the set of zeros of $\det(F(\a))$ in $\Cpos$. Hence the first
integral on the right converges and is equal to
$(\a\matI+\Lambda)^{-1}$ for $\a>\eta$. In addition,
Theorem~\ref{thm:transform} gives
\begin{equation}\label{eq:proof_transform}-F(\a)^{-1}=\int_0^\infty e^{\a x}\p[J(\tx{x})]\D
x\eL+\int_0^\infty e^{-\a x}\p[J(\tx{-x})]\D x\eL\end{equation} for
$\a\geq 0$ with $k(\a)<0$. Recall that $k(0)<0$ when $Q\1\neq \0$. The
continuity of $k(\a)$ implies that there exists $\epsilon>0$ such
that $k(\a)<0$ for all $\a\in[0,\epsilon)$. Equation~(\ref{eq:proof_transform}) can be rewritten as
\begin{equation}\label{eq:transform_positive}-\int_0^\infty e^{-\a x}\p[J(\tx{-x})]\D x\eL=F(\a)^{-1}-(\a\matI+\Lambda)^{-1}\eL\end{equation}
for $\a\in[0,\epsilon)$, because the real parts of all the
eigenvalues of $(\a\matI+\Lambda)$ are negative. The proof is
complete as soon as we show that the latter identity can be
continued to $\a>\eta$. To see this, multiply both sides by $F(\a)$
from the right and by $(\a\matI+\Lambda)$ from the
left. Then both sides are analytic for
$\a\in\Cpos$, hence
the equality holds for these $\a$. Finally, the matrices $F(\a)$ and $(\a\matI+\Lambda)$ are invertible for $\a>\eta$, so the above multiplication can be reversed.
\end{proof}

Let us proceed to the general case. Observe that (\ref{eq:W}) can be rewritten using (\ref{lambda}) and (\ref{eq:H}) as
\begin{equation}\label{eq:W_final}
W(x)=e^{-\Lambda x}\eL(x).
\end{equation}
This equation can be used as the definition of the scale matrix $W(x)$ in all the cases including the delicate case when $X$ oscillates. Recall that $\eL(x)$ is continuous in $x\geq 0$ and hence so is $W(x)$. Furthermore, $W(0)=\eL(0)$ which is given in~(\ref{eq:L0}).
Let us finish the proof of our main result.
\begin{proof}[Proof of Theorem~\ref{thm:main_scale}]
It is left to show that the displays in the statement hold for the case of no killing, and that $W(x)$ is invertible for all $x>0$.

The first part is based on a continuity argument.
Superimpose independent exponential killing of rate $q>0$, and denote the corresponding quantities through $F^q(\a),\Lambda^q,\eL^q(x),W^q(x)$. The first three objects converge to $F(\a),\Lambda,\eL(x)$ respectively as $q\downarrow 0$ and hence also $W^q(x)\rightarrow W(x)$. Take the limits in~(\ref{eq:exit}) and use the dominated convergence theorem to show that~(\ref{eq:exit}) also holds in the case of no killing.
Taking limits to obtain the transform is not that straightforward, because the entries of $W(x)$ are not necessarily non-negative (except when $X$ is a \levy processes where the extended continuity theorem for Laplace transforms could be applied).

First observe that $\eL^q(x)\leq \eL(x)$, where the inequality is entry-wise. Moreover, according to Lemma~\ref{lem:app_bound}
there exists $\lambda>0$ and $q_0>0$ such that the entries of the matrix $e^{-\lambda x}e^{-\Lambda^q x}$ are bounded in absolute value for all $x\geq 0$ and $q\in[0,q_0)$. Lemma~\ref{lem:L_transform} allows to apply the dominated convergence theorem to establish the transform of the scale matrix for $\a>\lambda$ and $q=0$. It is only left to show that this identity can be continued to all $\a>\eta$. Note that for $\a>\eta$ the entries of $e^{-\a x}e^{-\Lambda x}$ are bounded in absolute value for all $x\geq 0$; one can use the Jordan decomposition of $\Lambda$. But then Lemma~\ref{lem:L_transform} implies that $\int_0^\infty |e^{-\a x}W_{ij}(x)|\D x<\infty$ for any $i,j$. So the entries of the transform $\int_0^\infty e^{-\a x}W(x)\D x$ are analytic functions in the domain $\Re(\a)>\eta$, see~\cite[Ch.~II]{widder}. Hence the transform can be continued as stated above.

Let us show that $W(x)$ is invertible for all $x>0$. It is enough to establish invertibility of $\eL(x)$. For this observe
that the following relation is true for $y>0$ according to Proposition~\ref{prop:occ_density_MAP}:
\[\eL(x+y)=\eL(x)+\p[J(\tau_x^+)]\p[\tx{-x}<\tau^+_y,J(\tx{-x})]\eL(x+y).\]
Notice that  $\p_i(\tx{-x}<\tau^+_y)<1$, hence the matrix
$\matI-\p[J(\tau_x^+)]\p[\tx{-x}<\tau^+_y,J(\tx{-x})]$ is invertible
for any $x,y>0$, because it is strictly diagonally dominant, see~\cite{matrix_analysis}. Thus if for any $x>0$ there is a vector $\bv$,
such that $\eL(x)\bv=\0$ then $\eL(y)\bv=\0$ for all $y>0$. But then
$F(\a)^{-1}\bv=\0$ for $\a>\eta$, which is a contradiction.
\end{proof}

\begin{remark}
In the standard construction for \levy processes described in e.g.~\cite{bertoin} and \cite{kyprianou} one starts from considering no killing and process $X$ drifting to~$+\infty$.
Then $W(x)$ is taken proportional to $\p(\uX\geq -x)$.
On the contrary our construction works in all scenarios except one delicate case, when the process is recurrent.
If in our set-up we assume additionally that $q=0$ and $X$ drifts to~$+\infty$
then the expression in the brackets in~(\ref{eq:W}) simplifies to $1-\p(\tx{-x}<\infty)=\p(\tx{-x}=\infty)=\p(\uX>-x)$, which is known to coincide with $\p(\uX\geq -x)$.
\end{remark}


\section{Two-sided reflection}\label{sec:2_sided_reflection}
In this section we consider two-sided reflection of $X$ with respect to an interval $[-a,0]$, where $a>0$. The main result of this section, Theorem~\ref{thm:2_sided}, is a cornerstone in solving exit problems for processes $Y$ and $\hat Y$.
Assume that $X(0)=x\in[-a,0]$ and define the solution of Skorohod problem  $H(t)\in [-a,0]$ by
\[H(t)=X(t)+\Rl(t)-\Ru(t),\] where $\Rl$ and $\Ru$ are the regulators at the lower barrier $-a$
and at the upper barrier $0$ respectively.
That is, $\Rl$ and $\Ru$ are non-decreasing processes with $\Rl(0)=\Ru(0)=0$, such that their points of increase are contained in $\{t\geq 0:H(t)=-a\}$ and $\{t\geq 0:H(t)=0\}$ respectively. It is known that such a triplet $(H,\Rl,\Ru)$ exists and is unique, see~\cite{kruk_ramanan}.

Observe that the process $\Ru$ is continuous, because $X$ has no positive jumps. It can be seen as the local time of $H$ at $0$.
Let $\rho_r$ be its inverse:
\[\rho_r=\inf\{t\geq 0:\Ru(t)>r\}\text{ for }r\geq 0.\] In other words, $\rho_r$ is the first passage time of $\Ru$ over $r$.
Note also that $\Ru(\rho_r)=r$ and $H(\rho_r)=0$. The strong Markov property of $X$ shows that $(\Rl(\rho_r),J(\rho_r)),r\geq 0$ is a MAP itself.
In fact,
it is a Markov-modulated
compound Poisson process, because it has
no jumps in a fixed interval with positive probability. We denote its
matrix exponent through $F^*(\a)$,
that is, for $\a\geq 0$ and $r\geq 0$ it holds that
\begin{align*}\e[e^{-\a \Rl(\rho_r)};J(\rho_r)]=e^{F^*(\a)r}.\end{align*}
The aim of this section is to identify the law of $(\Rl(\rho_r),J(\rho_r))$ for any $X(0)\in[-a,0]$. This law is characterized in the following theorem.
\begin{theorem}\label{thm:2_sided}
For all $\a\geq 0$ and $x\in[-a,0]$ it holds that $Z(\a,a)$ is
invertible and
\begin{align}
F^*(\a)&=W(a)F(\a)Z(\a,a)^{-1}-\a\matI,\label{eq:2_1}\\
\e_{x}[e^{-\a\Rl(\rho_0)};J(\rho_0)]&=Z(\a,a+x)Z(\a,a)^{-1}.\label{eq:2_2}
\end{align}
\end{theorem}
\begin{proof}
It is enough to prove the theorem for a killed process; one can use a continuity argument as in the proof of Theorem~\ref{thm:main_scale} to extend the statement to the case of no killing. Hence we can assume that $Q\1\neq \0$.
From the definition of the two-sided reflection we have $0=X(\rho_r)+\Rl(\rho_r)-r$. So $(X(\rho_r),J(\rho_r))$ is a MAP such that its every underlying \levy process is of bounded variation with drift equal to~1.
Hence its occupation density, see Section~\ref{sec:occ_density}, is defined through
\[L^*(y,j,\infty)=\#\{r\geq 0:X(\rho_r)=y,J(\rho_r)=j\}.\]
The occupation density formula states that
\[\int_0^\infty e^{\a X(\rho_r)}\ind{J(\rho_r)=j}\D r=\int_\R e^{\a y}L^*(y,j,\infty)\D y.\]
Taking expectations on both sides we obtain
\[\int_0^\infty \e_{x}[e^{\a X(\rho_r)};J(\rho_r)]=\int_\R e^{\a y}\e_{x}L^*(y,\infty)\D y.\]
But the left hand side can be rewritten as
\[\int_0^\infty e^{\a r}\e_{x}[e^{-\a \Rl(\rho_r)};J(\rho_r)]=\e_{x}[e^{-\a \Rl(\rho_0)};J(\rho_0)]\int_0^\infty e^{\a r}e^{F^*(\a)r}\D r.\]
In the rest of the proof we will show that for small enough $\a\geq 0$ it holds that
\begin{equation}\label{eq:2_sided_proof}\int_\R e^{\a y}\e_{x}L^*(y,\infty)\D y=-Z(\a,a+x)F(\a)^{-1}W(a)^{-1}.\end{equation}
In particular, the integral on the left converges, which implies that $\int_0^\infty e^{\a r}e^{F^*(\a)r}\D r$ converges, and
\begin{equation}\label{eq:2_3}\e_{x}[e^{-\a \Rl(\rho_0)};J(\rho_0)](F^*(\a)+\a\matI)^{-1}=Z(\a,a+x)F(\a)^{-1}W(a)^{-1}.\end{equation} Note that $\rho_0=0$ a.s.\ when $x=0$ which leads to~(\ref{eq:2_1}).
Plugging~(\ref{eq:2_1}) into~(\ref{eq:2_3}) we obtain~(\ref{eq:2_2}).
Finally, analyticity arguments allow to continue these identities to all $\a\geq 0$.

We are only left to show that~(\ref{eq:2_sided_proof}) holds for small $\a\geq 0$.
Observe that $L^*(y,j,\infty)$ is equal to the number of time points $t\geq 0$, such that $X(t)=y,J(t)=j$ and $t$ is a point of increase of $\Ru$.
For $b\leq 0$ let
$$C_b=\p(J(\tx{b}))\p(J(\tau^+_a))$$ be the probability of hitting the level $b$ and after that passing through the level~$a+b$. One can show that
\begin{equation}\label{eq:N_y}\e_{x}L^*(y,\infty)=(\ind{y\geq 0}\p_{x}(J(\tau_y^+))
+\ind{y<0}C_{y-a-x})\sum_{i=0}^\infty (C_{-a})^i.\end{equation}
Indeed, if $y\geq 0$ then the first point counted by $L^*(y,\infty)$ is given by $\tau_y^+$, otherwise $X$ has to hit level $y-a$ to insure that at the next passage over $y$ the process $H$ is at its upper boundary (so that we get an increase point of $\Ru$). Every next point is obtained in a similar fashion: $X$ needs to descend to the level $y-a$ and then to pass through $y$. Here we also use the strong Markov property of~$X$.

 For a killed process we have $\sum_{i=0}^\infty (C_{-a})^i=(\matI-C_{-a})^{-1}$.
Take the transforms of the both sides of (\ref{eq:N_y}) to obtain
\begin{align*}\label{eq:inbrackets}\int_\R e^{\a y}\e_{x}L^*(y,\infty)\D y=\left(\int_0^\infty e^{\a y}e^{\Lambda(y-x)}\D y+\int_{-\infty}^0e^{\a y}\p(J(\tx{y-a-x}))\D y e^{\Lambda a}\right)(\matI-C_{-a})^{-1}.\end{align*}
The first term inside the large brackets is
$-(\Lambda+\a\matI)^{-1}e^{-\Lambda x}$ for small enough $\a\geq
0$. The second term is $e^{\a(a+x)}\int_{a+x}^\infty e^{-\a
y}\p(J(\tx{-y}))\D ye^{\Lambda a}$.
It is not difficult to see that for small $\a\geq 0$ it holds that
\[\int_x^\infty e^{-\a y}\p(J(\tx{-y}))\D
y=-F(\a)^{-1}\eL^{-1}+(\Lambda+\a\matI)^{-1}e^{-(\Lambda+\a\matI)x}+\int_0^xe^{-\a
y}W(y)\D y\eL^{-1};\]
for this integrate (\ref{eq:W}) over $(0,x)$ and use~(\ref{eq:transform_positive}).
Combine the last two displays to obtain
\[\int_\R e^{\a y}\e_{x}L^*(y,\infty)\D y=e^{\a(a+x)}\left(\int_0^{a+x}e^{-\a y}W(y)\D y-F(\a)^{-1}\right)\eL^{-1}e^{\Lambda a}(\matI-C_{-a})^{-1}.\]
In addition, the construction (\ref{eq:W}) shows that
$\matI-C_{-a}=W(a)\eL^{-1}e^{\Lambda a}$, and hence the term following large brackets reduces to $W(a)^{-1}$. So we finally get~(\ref{eq:2_sided_proof}).
\end{proof}
\section{Further exit problems}\label{sec:further}
The two-sided reflection, as considered in the previous section, is intimately related to the reflection at a single barrier.
In fact, Theorem~\ref{thm:2_sided} allows for a rather simple derivation of the exit identities given in Theorem~\ref{thm:reflected} and Theorem~\ref{thm:reflected2}.
\begin{proof}[Proof of Theorem~\ref{thm:reflected}]
It is enough to observe that $(H(t)+a,\Rl(t))$ up to time $\rho_0$ coincides with $(Y(t),R(t))$ up to $T_a$ given $H(0)+a=Y(0)$. The second identity of Theorem~\ref{thm:2_sided} yields the result.
\end{proof}

\begin{proof}[Proof of Theorem~\ref{thm:derivatives}]
Observe that for any $0<\epsilon<\delta$ it holds that
\[\p[\epsilon<\zeta_a,J(\tau_\epsilon^+)]\leq \p[\tau_\epsilon^+<\tau_a^-,J(\tau_\epsilon^+)]\leq \p[\epsilon<\zeta_{a+\delta},J(\tau_\epsilon^+)],\] where $\zeta_a$ is the (local) time when the first excursion of height exceeding $a$ arrives.
Subtract $\matI$, divide by $\epsilon$ and let $\epsilon\downarrow
0$ to obtain in view of~(\ref{eq:Lambda_a}) the bounds
\[\Lambda_{[0,a]}\leq\lim_{\epsilon\downarrow 0}(W(a)W(a+\epsilon)^{-1}-\matI)/\epsilon\leq \Lambda_{[0,a+\delta]}.\]
Deduce from probabilistic interpretation that $e^{\Lambda_{[0,a+\delta]}x}\rightarrow e^{\Lambda_{[0,a]}x}$ for every $x$ as $\delta\downarrow 0$. This implies
the convergence $\Lambda_{[0,a+\delta]}\rightarrow\Lambda_{[0,a]}$ which completes the proof of the first part of the theorem.

The proof of the second part follows the same steps. In this case we start by establishing the bounds:
\[e^{\Lambda_{[0,a-\delta]}\epsilon}\leq \p[\tau_\epsilon^+<\tau_{a-\epsilon}^-,J(\tau_\epsilon^+)]\leq e^{\Lambda_{[0,a)}\epsilon}.\]
\end{proof}

\begin{proof}[Proof of Theorem~\ref{thm:reflected2}]
Assume $X(0)=0$ and consider $\{(\Rl(\rho_r),J(\rho_r)),\; r\geq 0\}$, a Markov-modulated compound Poisson process with matrix exponent $F^*(\a)$.
Let \[\zeta=\inf\{r\geq 0:\Rl(\rho_r)>0\}\] be the epoch of the first jump of $\Rl(\rho_r)$.
Observe that $\Ru(t)$ coincides with $\oX(t)$ up to the time when $\Rl$ becomes positive (the lower barrier is hit). Hence $\zeta$ is the (local) time at which the first excursion from the maximum of height exceeding $a$ arrives. Use Lemma~\ref{lem:MMCPP} in Appendix to note that $\Lambda_{[0,a]}=F^*(\infty)$ and
\begin{align}\label{eq:transform_zeta}\e_0[e^{-\theta\zeta-\a\Rl(\rho_\zeta)};J(\rho_\zeta)]=\matI-(\Lambda_{[0,a]}-\theta\matI)^{-1}(F^*(\a)-\theta\matI)\end{align} for all $\theta\geq 0$ and $\a\geq 0$.

Suppose now that $X(0)=-x$, where $x\in[0,a]$, and let \[\varrho_0=\inf\{t\geq 0:\Rl(t)>0\}\] be the first increase point of $\Rl$. Observe that
\[\e_{\hat Y(0)=x}[e^{-\theta
\hat R(\hat T_a)-\a
(\hat Y(\hat T_a)-a)};J(\hat T_a)]=\e_{-x}[e^{-\theta
\Ru(\varrho_0)-\a
\Rl(\varrho_0)};J(\varrho_0)].\]
Hence we only need to calculate the latter. Recall that $\rho_\zeta$ is the first time the process $H(t)$ hits the upper barrier after it has hit the lower barrier.
Using the strong Markov property of $X$ we can write:
\[\e_{-x}[e^{-\theta \zeta-\a \Rl(\rho_\zeta)};J(\rho_\zeta)]=\e_{-x} [e^{-\theta \Ru(\varrho_0)-\a \Rl(\varrho_0)};J(\varrho_0)]\e_{-a} [e^{-\a \Rl(\rho_0)};J(\rho_0)].\]
Alternatively, this expectation can be computed by considering the
event $\{\tau_0^+<\tau_a^-\}$ and its complement:
\begin{align*}
\e_{-x}[e^{-\theta\zeta-\a \Rl(\rho_\zeta)};J(\rho_\zeta)]
=\p_{-x}[\tau_0^+<\tau_a^-,J(\tau_0^+)]\e_0[ e^{-\theta \zeta-\a
\Rl(\rho_\zeta)};J(\rho_\zeta)]\\+\e_{-x}[e^{-\a
\Rl(\rho_0)};\tau_a^-<\tau_0^+,J(\rho_0)].
\end{align*}
The last term reduces to $\e_{-x}[e^{-\a
\Rl(\rho_0)};J(\rho_0)]-\p_{-x}(\tau_0^+<\tau_a^-,J(\tau_0^+)).$
Combine the last two displays and use equation~(\ref{eq:transform_zeta}), Theorem~\ref{thm:main_scale} and Theorem~\ref{thm:2_sided} to obtain
\begin{align*}&\e_{-x} [e^{-\theta \Ru(\varrho_0)-\a \Rl(\varrho_0)};J(\varrho_0)]Z(\a,0)Z(\a,a)^{-1}\\&=W(a-x)W(a)^{-1}[\matI-(\Lambda_{[0,a]}-\theta\matI)^{-1}(W(a)F(\a)Z(\a,a)^{-1}-(\a+\theta)\matI)]\\&+Z(\a,a-x)Z(\a,a)^{-1}-W(a-x)W(a)^{-1}.\end{align*}
Note that $Z(\a,0)=\matI$ and $\Lambda_{[0,a]}=-W'_+(a)W(a)^{-1}$ according to Theorem~\ref{thm:derivatives} to finish the proof.
\end{proof}

\section*{Appendix}\label{sec:app}
\begin{proof}[Proof of Lemma~\ref{lem:L_transform}]
At the beginning we prove the second claim of this theorem.
Recall that by Remark \ref{po} the matrix $\eL$ has finite entries if $Q$ is transient.
Consider now the case case of no killing with $k'(0)\neq 0$. Then $\e_iL(0,j,e_q)=q\int_0^\infty
e^{-qt}\e_i L(0,j,t)\D t<\infty$. This further implies that $\e_i L(0,j,t)<\infty$
for any deterministic $t$. Let $\tau(j,t)=\inf\{s\geq t:X(s)=0,J(s)=j\}$. Note that $\p_j(\tau(j,t)<\infty)<1$ because by condition  $k'(0)\neq 0$ the process
$X$ drifts to $+\infty$ or to $-\infty$ (see \cite[Prop. XI.2.14]{asmussen:apq}).
It is enough to show now that $\e_jL(0,j,\infty)<\infty$. But the regenerative structure of $L$ implies that
\begin{equation}\label{eq:app1}\e_jL(0,j,\infty)\leq\sum_{n\geq 0}\p_j(\tau(j,t)<\infty)^n\e_jL(0,j,t),
\end{equation} which is finite when $k'(0)\neq 0$.

The proof of the first claim follows similar steps. Let $e_u$ be an exponential random variable of rate $u$ independent
of everything else. It is enough to show that $\e \eL(e_u)<\infty$.
Note that $\e \eL(e_u)$ is the expected occupation time at~0 up to
hitting the random level $e_u$. Hence the corresponding entries
of~$\e \eL(e_u)$ are bounded from above by the entries of $\eL$. So we only need to consider the case of no
killing with  $k'(0)=0$. In this case the counterpart of inequality~(\ref{eq:app1}) is
\[\e_jL(0,j,\tau^+_{e_u})\leq\sum_{n\geq 0}\p_j(\tau(j,t)<\tau^+_{e_u})^n\e_jL(0,j,t).\]
The proof is completed by observing that $\p_j(\tau(j,t)<\tau^+_{e_u})<1$ because $X$ oscillates.
\end{proof}

\begin{lemma}\label{lem:app_bound}
There exists a constant $C_n$ such that the elements of the matrix
$e^{-\lambda t}e^{-Qt}$ are bounded in absolute value by $C_n$ for
any $t\geq 0$, $n\times n$ transition rate matrix $Q$, and
$\lambda>-\sum_{i=1}^nq_{ii}$.
\end{lemma}
\begin{proof}
Let $P(t)=e^{Qt}$, so $e^{-\lambda t }e^{-Qt}=e^{-\lambda t
}P(t)^{-1}$. Observe that
\[\det(P(t))=\prod_{i=1}^n e^{\lambda_i t}=e^{{\rm tr
}(Q)t},\] where $\lambda_i$ are the eigenvalues of $Q$, and ${\rm tr
}(Q)$ is the trace of~$Q$. Moreover, the entries of transition
probability matrix $P(t)$ are bounded by 1, hence the cofactors of
$P(t)$ are bounded by $C_n$, a constant which only depends on~$n$.
According to Cramer's rule we get a bound $e^{-(\lambda+{\rm tr}(Q))
t}C_n$ for every element of the matrix $e^{-\lambda t}e^{-Qt}$. The
claim now follows immediately.
\end{proof}

\begin{lemma}\label{lem:MMCPP}
Let $(X,J)$ be a Markov-modulated
compound Poisson process without
negative jumps, such that $\e[e^{-\a X(t)};J(t)]=e^{F(\a)t}$. Let also $T=\{t\geq 0:X(t)\neq 0\}$ be the epoch of the first jump of
$X$. Then $J(t)$ killed at $t=T$ is a Markov chain with transition rate matrix $F(\infty)$. Moreover, for $q\geq 0$ and $\a\geq 0$ it holds that
\[\e [e^{-qT-\a X(T)};J(T)]=\matI-(F(\infty)-q\matI)^{-1}(F(\a)-q\matI).\]
\end{lemma}
\begin{proof}
Observe that
$\e[e^{-\a
X(e_q)};J(e_q)]=q\int_0^\infty e^{(F(\a)-q\matI)t}\D t$, which
converges and is equal to $-q(F(\a)-q\matI)^{-1}$ for all $q>0$
and $\a\geq 0$.
Furthermore,
\[\lim_{\a\rightarrow\infty}\e[e^{-\a
X(e_q)};J(e_q)]=\p[X(e_q)=0,J(e_q)]=\p[e_q<T,J(e_q)]=-q(\Lambda-q\matI)^{-1},\]
where $\Lambda$ is the transition rate matrix of $J$ killed at the first jump of $X$.
This shows that $F(\infty)=\Lambda$.

Using the strong Markov property and memoryless property of $e_q$ we
write
\[\e[e^{-\a X(e_q)};J(e_q)]=\p[e_q<T,J(e_q)]+\e[e^{-\a
X(T)};e_q>T,J(T)]\e[e^{-\a X(e_q)};J(e_q)].\] Use the above observations to complete the proof for $q>0$. Let $q\downarrow 0$ to obtain the result for $q=0$.
\end{proof}

\end{document}